\documentclass[12pt]{article}

\usepackage{amssymb,amsthm,float,fullpage}

\author{Jean-Paul Allouche \\
CNRS, Institut de Math\'ematiques \\
Universit\'e Pierre et Marie Curie \\
4 place Jussieu, F-75752 Paris Cedex 05 \\
France \\
{\tt allouche@math.jussieu.fr} \\
\and
Jeffrey Shallit\\
School of Computer Science\\
University of Waterloo\\
Waterloo, Ontario N2L 3G1\\
Canada\\
{\tt shallit@cs.uwaterloo.ca} \\
}

\title{A variant of Hofstadter's sequence and finite automata}

\date{ }

\def \endpf{{\ \ $\Box$ \medbreak}}

\theoremstyle{plain}
\newtheorem{theorem}{Theorem}
\newtheorem{corollary}[theorem]{Corollary}
\newtheorem{lemma}[theorem]{Lemma}

\theoremstyle{definition}
\newtheorem{definition}[theorem]{Definition}

\theoremstyle{remark}
\newtheorem{remark}[theorem]{Remark}

\begin{document}

\maketitle

\vskip .2in

\begin{center}
{\it In memory of Alf van der Poorten:  colleague, connoisseur, raconteur,
friend}
\end{center}

\vskip .2in

\begin{abstract}
Following up on a paper of Balamohan, Kuznetsov, and Tanny,
we analyze a variant of Hofstadter's $Q$-sequence and show it is
$2$-automatic.  An automaton computing the sequence is explicitly given.
\end{abstract}

\section{Introduction}

In his 1979 book {\it G\"odel, Escher, Bach} \cite{Hof}, Douglas
Hofstadter introduced the sequence $Q(n)$ defined by the recursion $$
Q(n) = Q(n-Q(n-1)) + Q(n-Q(n-2))$$ for $n \geq 2$ and $Q(1) = Q(2) =
1$.  Although it has been studied extensively (e.g., \cite{Pinn:1999}),
still little is known about its behavior, and it is not mentioned
in standard books about recurrences (e.g., \cite{EPSW}).  It is
sequence A005185 in Sloane's {\it Encyclopedia} \cite{Sloane}.

Twenty years later, Hofstadter and Huber introduced a family of
sequences analogous to the $Q$-sequence, and defined by the recursion
$$ Q_{r,s}(n) = Q_{r,s} (n-Q_{r,s} (n-r)) + Q_{r,s} (n-Q_{r,s} (n-s))
$$ for $n > s > r$ \cite{HH}.  The case $r = 1$, $s = 4$ is of
particular interest.

Recently Balamohan, Kuznetsov and Tanny \cite{BKT} gave a nearly complete
analysis of the sequence $Q_{1,4}$ (called $V$ in their paper).
It is defined by
$$
V(1) = V(2) = V(3) = V(4) = 1, \ \mbox{\rm and \ }
\forall n > 4, \ V(n) := V(n-V(n-1)) + V(n-V(n-4)).
$$
Here is a short table of the sequence $V$ 
(sequence A063882 in Sloane's {\it Encyclopedia} \cite{Sloane}).

\begin{table}[H]
\begin{center}
\begin{tabular}{|c|c|c|c|c|c|c|c|c|c|c|c|c|c|c|c|c|c|c|c|c|}
\hline
$n$ & 1 & 2 & 3 & 4 & 5 & 6 & 7 & 8 & 9 & 10 & 11 & 12 & 13 & 14 & 15 &
	16 & 17 & 18 & 19 & 20 \\
\hline
$V(n)$ & 1 & 1 & 1 & 1 & 2 & 3 & 4 & 5 & 5 & 6 & 6 & 7 & 8 & 8 & 9 & 9 & 
	10 & 11 &  11 &  11 \\
\hline
\end{tabular}
\end{center}
\end{table}

Among the results of Balamohan, Kuznetsov, and Tanny
is a precise description of
the ``frequency'' sequence $F(n)$ defined by
$$
F(a) := \# \{n, \ V(n) = a\}.
$$

Here is a short table of the sequence $F$ (sequence A132157 in Sloane's 
{\it Encyclopedia} \cite{Sloane}).
\begin{table}[H]
\begin{center}
\begin{tabular}{|c|c|c|c|c|c|c|c|c|c|c|c|c|c|c|c|c|c|c|c|c|}
\hline
$n$ & 1 & 2 & 3 & 4 & 5 & 6 & 7 & 8 & 9 & 10 & 11 & 12 & 13 & 14 & 15 &
	16 & 17 & 18 & 19 & 20 \\
\hline
$F(n)$ & 4 & 1 & 1 & 1 & 2 & 2 & 1 &  2 &  2 &  1 &  3 &  2 & 
	1 &  2 &  2 &  1 &  3 &  2 &  1 &  2 \\
\hline
\end{tabular}
\end{center}
\end{table}

In particular they proved the following theorem
\cite[Lemmas 13--19 and Table 5]{BKT}.

\begin{theorem}[Balamohan, Kuznetsov, Tanny]
There exist two (explicit) maps $g, h$, with
$g,h : \{1, 2, 3\}^4 \to \{1, 2, 3\}$,
such that, for all $a > 3$
$$
\begin{array}{cll}
F(2a) &=& g(F(a-2), F(a-1), F(a), F(a+1)) \\
F(2a+1) &=& h(F(a-2), F(a-1), F(a), F(a+1)).
\end{array}
$$
\end{theorem}

\noindent
(We note that in Lemma 13 of \cite{BKT}, the quantifiers
$a \geq 3$ for the equality $F(2a) = 2$ and $a \geq 4$ for the
equality $F(2a+1) = 2$ should have been mentioned.)

\bigskip

In this paper we prove that the sequence 
$(F(n))_{n \geq 1}$ is $2$-automatic, which means essentially
that $F(n)$ can be computed ``in a simple way'' from the
base-$2$ representation of $n$ --- in particular, it can be computed 
in $O(\log n)$ time.
Furthermore, we give the automaton explicitly.
For definitions and properties
of automatic sequences, the reader is referred to \cite{AS}.
For some recent related papers, see \cite{IRT,DRT,Rahman}.

\section{The main result}

We begin this section with a general result on automatic sequences.
Before stating the theorem we need a notation.

\begin{definition}
Let $W = (W(n))_{n \geq 0}$ be a sequence. Let $\alpha$ be an integer
in ${\mathbb Z}$. We let $W^{\alpha}$ denote the sequence defined, for
$n \geq - \alpha$, by
$$
W^{\alpha}(n) := W(n + \alpha).
$$
\end{definition}

\begin{definition}
Let $c$ be an integer $\geq 0$.
If a sequence $(W(n))_{n \geq 0}$ is only defined
for $n \geq c$, we assume that the values of $W(n)$ for $n \in [0, c)$
are arbitrary.
\end{definition}

\begin{theorem}\label{main}
Let $(U(n))_{n \geq 0}$ be a sequence with values in a finite set 
${\cal A}$. Let $q \geq 2$ be an integer. Suppose that there exist 
four nonnegative integers $t, a, b, n_0$, and $q^{t+1}$ functions from the 
set ${\cal A}^{a+b+\frac{q^{t+1}-1}{q-1}}$ to ${\cal A}$, denoted
$f_0, f_1, \ldots, f_{q^{t+1}-1}$, such that 
$\forall j \in [0, q^{t+1}-1]$ and $\forall n \geq n_0$ 
$$
\begin{array}{ll}
U(q^{t+1}n+j) = \\
\ \ \ \
f_j(U^{-a}(n), \ldots, U^{-1}(n), U^0(n), U^{1}(n), \ldots, U^b(n),
U_2(n), U_3(n), \ldots, U_{\frac{q^{t+1}-1}{q-1}}(n))
\end{array}
$$
where $U_1 = U, U_2, \ldots, U_{\frac{q^{t+1}-1}{q-1}}$ are 
the subsequences $(U(q^i n + j))_{n \geq 0}$ with $i \in [0,t]$ 
and $j \in [0, q^i-1]$, written in some fixed order.
Then the sequence $(U(n))_{n \geq 0}$ is $q$-automatic.
\end{theorem}

Before proving this theorem we recall the Euclidean division of
an integer in ${\mathbb Z}$ by a positive integer.

\begin{lemma}
Let $S$ be an integer in ${\mathbb Z}$ and $Q$ be a positive integer.
Then there exist $X \in {\mathbb Z}$ and an integer
$Y$, $0 \leq Y < Q$, such that 
$ S = QX+Y  $.
\end{lemma}

\begin{proof}
Let $X = \lfloor S/Q \rfloor$ and $Y = S - Q \lfloor S/Q \rfloor$.
Then $0 \leq S/Q - \lfloor S/Q \rfloor < 1$, so,
multiplying by Q, we get
$0 \leq Y < Q$. 
\endpf
\end{proof}

\bigskip

\noindent
{\it Proof of Theorem~\ref{main}}.
To prove that the sequence $U = (U(n))_{n \geq 0}$ is $q$-automatic, it
suffices to find a finite set of sequences ${\cal E}$ that contains $U$, 
such that if $V = (V(n))_{n \geq 0}$ belongs to ${\cal E}$, then, for
any $r \in [0, q-1]$ the sequence $(V(qn+r))_{n \geq 0}$ also belongs to
${\cal E}$. 
Fix two positive integers $A$ and $B$ such that
$A \geq \max(n_0, \frac{q(a+1)}{q-1})$ and $B \geq \frac{q(b+1)}{q-1}$.
Recall that $U_1=U$, $U_2, \ldots, U_{\frac{q^{t+1}-1}{q-1}}$ are the 
sequences $(U(q^i n + j))_{n \geq 0}$ with $i \in [0,t]$ and 
$j \in [0, q^i-1]$. Also recall that the sequence $U_k^{\alpha}$ is
defined by $U_k^{\alpha}(n) := U_k(n+\alpha)$. Let ${\cal E}$ be the 
(finite) set of sequences defined by
$$
V \in {\cal E} \Longleftrightarrow \exists \ell \in [1, \frac{q^{t+1}-1}{q-1}], 
\ \exists k \in [-A,B], \ \forall n \geq A, \ V(n) = U_{\ell}^k(n).
$$
Now let $V$ be a sequence in ${\cal E}$. Take $r \in [0, q-1]$. 
There exist $\ell \in [1, \frac{q^{t+1}-1}{q-1}]$ and $k \in [-A,B]$ 
such that for all $n \geq A$, we have
$$
V(qn+r) = U_{\ell}^k(qn+r) = U_{\ell}(qn+r+k).
$$
Hence for some $i \leq t$ and $j \in [0, q^i-1]$
$$
V(qn+r) = U(q^i(qn+r+k)+j).
$$
Write $q^i(r+k)+j = q^{i+1}x + y$, with $x \in {\mathbb Z}$ and 
$y \in [0, q^{i+1}-1]$, so that 
$$
V(qn+r) = U(q^{i+1}(n+x) + y).
$$
Note that 
$$
q^{i+1}x \leq q^{i+1}x + y = q^i(r+k)+j < q^i(r+k+1)
$$
and
$$
q^{i+1}x = q^i(r+k)+j - y > q^i(r+k) - q^{i+1}
$$
Hence
$$
\frac{r+k-q}{q} < x < \frac{r+k+1}{q}\cdot
$$
We distinguish two cases.

\bigskip

\noindent{\it Case 1:}   $i < t$. Then $i+1 \leq t$. Thus there exists 
$\ell' \in [1, \frac{q^{t+1}-1}{q-1}]$ such that, for $n \geq A$,
$$
V(qn+r) = U(q^{i+1}(n+x) + y) = U_{\ell'}(n+x) = U_{ell'}^x(n).
$$
Now $x > \frac{r+k-q}{q} \geq \frac{r-A-q}{q} \geq \frac{-A-q}{q} \geq -A$
(since $A \geq \frac{q(a+1)}{q-1} \geq \frac{q}{q-1}$), and 
$x < \frac{r+k+1}{q} \leq \frac{q+B}{q} \leq B$ 
(since $B \geq \frac{q(b+1)}{q-1} \geq \frac{q}{q-1}$).
This shows that the sequence $(V(qn+r))_{n \geq 0}$ belongs to ${\cal E}$.

\bigskip

\noindent{\it Case 2:}   $i=t$. Then $i+1=t+1$. From the hypothesis and the condition 
$A \geq n_0$, we can write, for $n \geq A$,
$$
\begin{array}{lll}
V(qn+r) = U(q^{t+1}(n+x) + y) = \\
\ \ \ \
f_y(U^{x-a}(n), \ldots, U^{x-1}(n), U^x(n), U^{x+1}(n), \ldots, U^{x+b}(n),
U_2^x(n), U_3^x(n), \ldots, U_{\frac{q^{t+1}-1}{q-1}}^x(n)).
\end{array}
$$
To prove that the sequence $(V(qn+r))_{n \geq 0}$ belongs to ${\cal E}$,
it suffices to prove that all sequences $U^{\beta}$ for $\beta 
\in [x-a, x+b]$ and all sequences $U_{\ell}^x$ for
$\ell \in [1, \frac{q^{t+1}-1}{q-1}]$ belong to ${\cal E}$, and to use 
composition of maps.
But we have
$$
\beta \geq x-a > \frac{r+k-q}{q} - a \geq \frac{-A-q}{q} - a \geq -A
$$
(recall that $A \geq \frac{q(a+1)}{q-1}$) and
$$
\beta \leq x+b < \frac{r+k+1}{q} + b \leq \frac{q+B}{q} + b \leq B 
$$
(recall that $B \geq \frac{q(b+1)}{q-1}$).
This implies that all sequences occurring in the arguments of $f_y$
above belong to ${\cal E}$. \endpf

\begin{remark} 
Theorem~\ref{main} above is similar to (but different from)
\cite[Theorem 6, p.\ 5]{AS2} on $k$-regular sequences. That theorem
implies Theorem~\ref{main} above in the case where the maps $f_j$ are linear.
\end{remark}

\begin{corollary}
The sequence $F = (F(n))_{n \geq 0}$ 
is $2$-automatic.
\end{corollary}

\begin{proof}
It suffices to use the theorem recalled in the first section, after
having extended the sequence $F$ by $F(0)=0$.  \endpf
\end{proof}

\section{An explicit automaton}

In this section we provide an explicit automaton\footnote{In honor
of Alf van der Poorten, we cannot resist quoting Voltaire:
``Impuissantes machines/
Automates pensants mus par des mains divines."}
to calculate the
sequence $F$.

\bigskip

The automaton is constructed in two stages. First, we give an automaton
$A$ with the property that reading $n$ in base $2$ 
takes us to a state $q$ with the property that
the four values $F(n+a)$ for $-2 \leq a \leq 1$ are completely determined
by $q$.  Next, we show that $A$ can be minimized to give an automaton $B$
computing $F(n)$.  We remark that we assume throughout that the automaton
reads the ordinary base-$2$ representation
of $n$ from ``left to right'', ending at the
least significant digit, although we do allow the possibility of leading
zeros at the start.

Let us start with the description of $A = (Q, \Sigma, \Delta,
\delta, q_0, \tau)$.    The machine $A$ has 33 states with strings as
names; $\Sigma = 
\lbrace 0, 1 \rbrace$; $\Delta = \lbrace 0,1,2,3, 4 \rbrace^4$,
$q_0 = \epsilon$.  The transition function $\delta$ and the output
map $\tau$ are given in Table~\ref{table1} below.

We introduce some notation.
Let $[w]$ denote the integer represented by the binary string
$w$ in base $2$.  Thus, for example, $[00110] = [110] = 6$.
Note that $[\epsilon] = 0$, where
$\epsilon$ denotes the empty string.  If $F$ is our sequence defined above,
then by $F(a..a+i-1)$ we mean the string of length $i$ given by the
values of the function $F$ at $a, a+1, \ldots, a+i-1$.

Our intent is that if $w$ is a binary string, then
$\tau(\delta(q_0, w))$ is the string of length $4$ given by 
$F(n-2..n+1)$, where
$n = [w]$.    (Note that we define $F(0) = F(-1) = F(-2) = 0$.)

To prove that this automaton computes $F(n)$ correctly,
it suffices to show that
\begin{itemize}
\item[(a)] for each state $q$ we have
$\tau(q) = F([q]-2) F([q]-1) F([q]) F([q]+1)$; and

\item[(b)]  if $p = \delta(q,a)$ for two states $p, q \in Q$ and
$a \in \lbrace 0, 1 \rbrace$, then
$F([px]) = F([qax])$ for all strings $x$.
\end{itemize}

Part (a) can be verified by a computation, which we omit.
For example, since $[111001111] = 463$, the claim
$\tau(111001111) = 2133$ means $F(461..464) = 2133$, which can easily
be checked.

Part (b) requires a tedious simultaneous induction on all the assertions,
by induction on $|x|$.
Not surprisingly, we omit most of the details and just prove
a single representative case.

Consider the transition $\delta(100,1) = 110$.
Here we must prove that 
\begin{equation}
F([1001x]) = F([110x])
\label{1001}
\end{equation}
for all strings $x$.  We
do so by induction on $x$.
The base case is $x = \epsilon$, and we have
$F([1001]) = F(9) =2$ and $F([110]) = F(6) = 2$.  

For the induction step, we use the fact that 
\cite[Table 5]{BKT} shows that $F(2a)$ and $F(2a+1)$ is completely
determined by $F(a-2)$, $F(a-1)$, $F(a)$, and $F(a+1)$.  It thus
suffices to check that $F([1001x] + a) = F([110x] + a)$ for 
$-2 \leq a \leq 1$; doing so will then prove (\ref{1001}) for
$x0$ and $x1$, thus completing the induction.

The only cases that require any computation
are when $[x] = 0$ and $a = -1, -2$,
or $[x] = 1$ and $a = -2$, or
$x$ is a number of the form $2^j-1$ for some
$j \geq 1$ and $a = 1$.

\noindent{\it Case 1:}   $x = 0^j$ for some $j \geq 0$.  If $j = 0$ then 
this is the assertion that $F([1001] + a) = F([110] + a)$ for
$-2 \leq a \leq 1$, which is the same as the claim that
$F(7..10) = F(4..7)$.  But $F(7..10) = 1221 = F(4..7)$.

\begin{table}[H]
\caption{The automaton A}
\label{table1}
\begin{center}
\begin{tabular}{|c|c|c|c|}
\hline
$q$ & $\delta(q,0)$ & $\delta(q,1)$ & $\tau(q)$ \\
\hline
$\epsilon$ & $\epsilon$ & 1 & 0004 \\
\hline
1 & 10 & 11 & 0041 \\
\hline
10 & 100 & 101 & 0411 \\
\hline
11 & 110 & 111 & 4111 \\
\hline
100 & 1000 & 110 & 1112 \\
\hline
101 & 1010 & 1011 & 1122 \\
\hline
110 & 1100 & 1101 & 1221 \\
\hline
111 & 1110 & 110 & 2212 \\
\hline
1000 & 1010 & 1011 & 2122 \\
\hline
1010 & 1110 & 10101 & 2213 \\
\hline
1011 & 10110 & 10111 & 2132 \\
\hline
1100 & 1101 & 1110 & 1321 \\
\hline
1101 & 11010 & 11011 & 3212 \\
\hline
1110 & 11100 & 11101 & 2122 \\
\hline
10101 & 101010 & 101011 & 1223 \\
\hline
10110 & 10110 & 10111 & 2232 \\
\hline
10111 & 1101 & 1110 & 2321 \\
\hline
11010 & 101010 & 110101 & 1222 \\
\hline
11011 & 111 & 1000 & 2221 \\
\hline
11100 & 11010 & 111001 & 2213 \\
\hline
10111 & 111010 & 10111 & 2132 \\
\hline
101010 & 1010100 & 11101 & 1322 \\
\hline
101011 & 101010 & 101011 & 3223 \\
\hline
110101 & 1100 & 1101 & 3221 \\
\hline
111001 & 1010 & 1110011 & 2223 \\
\hline
111010 & 110100 & 1110101 & 2232 \\
\hline
1010100 & 11010 & 111001 & 3213 \\
\hline
1110011 & 10110 & 11100111 & 2133 \\
\hline
1110100 & 11101000 & 10111 & 1332 \\
\hline
1110101 & 1101 & 1110 & 3321 \\
\hline
11100111 & 1010100 & 111001111& 2323 \\
\hline
11101000 & 110100 & 1110101&  3232 \\
\hline
111001111 & 111010 & 11100111 & 2133 \\
\hline
\end{tabular}
\end{center}
\end{table}

Otherwise $j \geq 1$.  Then $[1001x] - 1 = [1001 0^j] - 1 =
[1000 1^j]$  and $ [110x] - 1 = [1100^j] - 1 = [1011^j]$.
Now by induction we have
$F([10001^j]) = F([10001 1^{j-1}]) = F([1011 1^{j-1}]) =
F([1011^j])$, as desired.

Similarly, $[1001x] - 2 = [1001 0^j ] - 2 = [1001^{j-1} 0 ]$.
Also $[110x] - 2 = [1100^j] - 2 = [101^{j} 0]$.
Then by induction we have $F([1001^{j-1}0]) = 
F([1 0 0 1 1^{j-2} 0]) = F([1011 1^{j-2} 0]) = F([101^j 0])$, as
desired.

\noindent{\it Case 2:}  $x = 0^j 1$ for some $j \geq 0$.  Then
$[1001x] - 2 = [10010^j 1] - 2 = [10001^{j+1}]$.
Also $[110x] - 2 = [110 0^j 1] - 2 = [101^{j+2}]$.  
By induction we have $F([10001^{j+1}]) = F([10001 1^j]) =
F([1011 1^j]) = F([101^{j+2}]$, as desired.

\noindent{\it Case 3:}  $x = 1^j$ for some $j \geq 1$.  Then
$[1001x] + 1 = [1010^{j+1}]$.  Similarly
$[110x] + 1 = [110 1^j] + 1 = [1110^j]$.
By induction we have $F([1010^{j+1}]) = F([10100 0^{j-1}])
= F([1110 0^{j-1}]) = F([1110^j])$, as desired.

This completes the proof of correctness of a single transition.  

Ultimately, we are not really interested in computing $\tau(q)$, but
only the image of $\tau(q)$ formed by extracting the third component,
which is the one corresponding to $F(n)$.  This means that we can
replace $\tau$ by $\tau'$, which is the projection of $\tau$ along
the third component.  In doing so some of the states of $A$ become
equivalent to other states.  We can now use the standard minimization
algorithm for automata to produce the 20-state minimal automaton
$B = (Q', \Sigma, \Delta, \delta', q_0, \tau')$  computing
$F(n)$.  Table~\ref{table2} below gives the names of the states of
$Q$, and $\delta'$ and $\tau'$ for these states.

\section{Concluding remarks}

It would be interesting to know whether the first difference sequence of 
the variant of Hofstadter's, i.e., the sequence $(V(n+1)-V(n))_{n \geq 0}$,
is also $2$-automatic. We already know that it takes only finitely many
values \cite[Theorem~1, page~5]{BKT}. Of course it might well be the case
that this sequence is {\em not\,} automatic: in a very different context,
think of the classical Thue-Morse sequence which is $2$-automatic, but whose
runlength sequence is not \cite{AAS}. It would be also interesting to
determine for which sequences $Q_{r,s}$ (with the notation in the introduction)
the frequency sequence is automatic.

\begin{table}[H]
\caption{The automaton B}
\label{table2}
\begin{center}
\begin{tabular}{|c|c|c|c|}
\hline
$q$ & $\delta'(q,0)$ & $\delta'(q,1)$ & $\tau'(q)$ \\
\hline
$\epsilon$ & $\epsilon$ & 1 & 0 \\
\hline
1 & 10 & 11 & 4 \\
\hline
10 & 100 & 101 & 1 \\
\hline
11 & 110 & 111 & 1 \\
\hline
100 & 101 & 110 & 1 \\
\hline
101 & 1010 & 1011 & 2 \\
\hline
111 & 1110 & 110 & 1 \\
\hline
1010 & 1110 & 10101 & 1 \\
\hline
1011 & 1011 & 1100 & 3 \\
\hline
1100 & 1101 & 1110 & 2 \\
\hline
1101 & 11010 & 11011 & 1 \\
\hline
1110 & 11100 & 1011 & 2 \\
\hline
10101 & 1110 & 10101 & 2 \\
\hline
11010 & 1110 & 110 & 2 \\
\hline
11011 & 111 & 101 & 2 \\
\hline
11100 & 11010 & 111001 & 1 \\
\hline
111001 & 1010 & 1110011 & 2 \\
\hline
110011 & 1011 & 11100111 & 3 \\
\hline
11100111 & 11100 & 1110011 & 2 \\
\hline
\end{tabular}
\end{center}
\end{table}

\end{document}